\documentclass[reqno,11pt]{amsart}

\usepackage[T1]{fontenc} 

\usepackage{amssymb}
\usepackage{amsmath}
\usepackage{amsfonts}
\usepackage{amsthm}
\usepackage{mathtools}
\usepackage{mathabx}

\usepackage{fullpage}  
\usepackage{enumerate}

\usepackage{url} 
\usepackage{hyperref}
 
\theoremstyle{plain}%
\newtheorem{theorem}{Theorem}[section]
\newtheorem{lemma}[theorem]{Lemma}

\newtheorem{conjecture}[theorem]{Conjecture}
\newtheorem*{conjecture*}{Conjecture}

\theoremstyle{definition}
\newtheorem{definition}[theorem]{Definition}

\theoremstyle{remark}
\newtheorem{remark}[theorem]{Remark}
\newtheorem{remarks}[theorem]{Remarks}

 \let \leq \leqslant
 \let \geq \geqslant

\DeclareMathOperator{\sgn}{sgn}

\DeclareMathOperator{\Var}{Var}

\title{A tame sequence of transitive Boolean functions}

\author{Malin P. Forsstr\"om}
\address[Malin P. Forsstr\"om]{Department of Mathematics, KTH Royal Institute of Technology, 100 44 Stockholm, Sweden.}
\email{malinpf@kth.se} 
 
  \keywords{Boolean functions}
  \subjclass[2010]{60K99}

\begin{document}

\begin{abstract}
Given a sequence of Boolean functions \( (f_n)_{n \geq 1} \),  \( f_n \colon \{ 0,1 \}^{n}  \to \{ 0,1 \}\), and a sequence \( (X^{(n)})_{n\geq 1} \) of continuous time \( p_n \)-biased random walks \( X^{(n)} =  (X_t^{(n)})_{t \geq 0}\) on \( \{ 0,1 \}^{n} \), let \( C_n \) be the (random) number of times in \( (0,1) \) at which the process  \( (f_n(X_t))_{t \geq 0} \) changes its value. In~\cite{js2006}, the authors conjectured that if \( (f_n)_{n \geq 1} \) is    non-degenerate, transitive and satisfies \( \lim_{n \to \infty}  \mathbb{E}[C_n] = \infty\), then \( (C_n)_{n \geq 1} \) is not tight. We give an explicit example of a sequence of Boolean functions which disproves   this conjecture.
\end{abstract}

\maketitle

\section{Introduction}

The aim of this paper is to present an example of a sequence of Boolean functions, which show that a conjecture made in~\cite{js2006}, in its full generality, is false. To be able to present this conjecture, we first give some background.

For each \( n \geq 1 \), fix some \( p_n \in (0,1) \) and let \(X^{(n)} = (X_t^{(n)})_{t \geq 0}\) be the continuous time random walk on the \( n \)-dimensional hypercube defined as follows.  For each \( i \in [n] \coloneqq \{1,2, \ldots, n \} \) independently, let \( (X_t^{(n)}(i))_{t \geq 0} \) be the continuous time Markov chain on \( \{ 0,1 \} \) which at random times, distributed according to a rate one Poisson process, is assigned a new value, chosen according to   \(  (1-p_n) \delta_0  + p_n \delta_1   \), independently of the Poisson process. The unique stationary distribution of $(X^{(n)}_t)_{t\geq  0}$, denoted by  \( \pi_n\), is the   measure \( ((1-p_n) \delta_0  + p_n \delta )^{\otimes n} \) on \( \{0,1 \}^n \). Throughout this paper, we will always assume that  \( X_0^{(n)}   \) is chosen with respect to this measure. When  \( t > 0 \) is small, the difference between \( X_0^{(n)} \) and \( X_t^{(n)} \) is often thought of as noise, describing a small proportion \( 1-e^{-t} \approx t\)  of the bits being miscounted or corrupted.

A function \( f_n \colon \{ 0,1 \}^{n} \mapsto \{0,1\} \) will be referred to as a Boolean function.  Some classical examples of Boolean functions are the so called \emph{Dictator function} \( f_n(x) = x(1) \), the \emph{Majority function} \( f_n(x) = \sgn \bigl(\sum_{i=1}^{n} (x(i) - 1/2) \bigr) \) and the \emph{Parity function} \( \sgn  \bigl( \prod_{i=1}^{n} (x(i)-1/2 ) \bigr)\)  (see e.g.~\cite{odonnell, gs}).
 Since it is sometimes not natural to require that a sequence of Boolean functions is defined for each \( n \in \mathbb{N} \), we only require that a sequence of Boolean functions is defined for \( n \) in an infinite sub-sequence of \( \mathbb{N} \). Such sub-sequences of \( \mathbb{N} \) will be denoted by \( \mathcal{N} = \{ n_1,n_2, \ldots \} \), where \( 1 \leq n_1 < n_2 < \ldots \). To simplify notation, whenever we consider the limit of a sequence \( (x_{n_i})_{i \geq 1} \) and the dependency on \( \mathcal{N} \) is clear, we will abuse notation and write \( \lim_{n \to \infty} x_n \) instead of \( \lim_{i \to \infty} x_{n_i}\).  Also, we will write \( (x_n)_{n \in \mathcal{N}} \) instead of \( (x_{n_i})_{i \geq 1} \).

One of the main objectives of~\cite{js2006} was to introduce notation which describes possible behaviours of \( (f_n(X_t^{(n)}))_{t \geq 0} \). Some of these definitions which will be relevant for this paper is given in the following definition.

\begin{definition}
Let \( (f_n)_{n \in \mathcal{N}}\), \( f_n \colon \{ 0,1 \}^{n} \to \{ 0,1 \} \), be a sequence of Boolean functions. For \( n \in \mathcal{N} \), let \( C_n = C_n(f_n) \) denote the (random) number of times in \( (0,1) \) at which \( (f_n(X_t^{(n)}))_{t \geq 0} \) has changed its value, i.e.\ let   \( C_n \coloneqq  \lim_{N \to \infty}  \sum_{i=0}^{N-1} \mathbf{1}({f_n(X^{(n)}_{i/N}) \neq f_n(X^{(n)}_{(i+1)/N})}) \).
The sequence \( (f_n)_{n \in \mathcal{N}} \) is said to be
\begin{enumerate}[(i)]
	\item \emph{lame} if
	\begin{equation*}
		\lim_{n \to \infty} P(C_n=0) = 1,
	\end{equation*}
	\item \emph{tame} if   \( (C_n)_{n\geq 1} \) is tight, that is for every \( \varepsilon > 0 \) there is \( k \geq 1 \) and \( n_0 \geq 1 \) such that 
	\begin{equation*}
		   P(C_n\geq k) < \varepsilon \qquad \forall n \in \mathcal{N} \colon n \geq n_0,
	\end{equation*}
	\item \emph{volatile} if \( C_n \Rightarrow \infty \) in distribution.
\end{enumerate}
\end{definition}

In~\cite{js2006}, the authors showed that a sequence of Dictator functions is tame and that a sequence of Parity functions is volatile, while a sequence of Majority functions is neither tame nor volatile. More generally, the authors also showed that any noise sensitive sequence of Boolean functions (see e.g.~\cite{bks, gs, odonnell}) is volatile, while any sequence of Boolean functions which is lame or tame is noise stable~\cite{bks, gs, odonnell}. As noted in~\cite{f}~and~\cite{pg}, there are many sequences of functions which are both noise stable and volatile, and hence the opposite does not hold.

Given a sequence \( (p_n)_{n \geq 1} \), \( p_n \in (0,1) \), a sequence of Boolean functions \( (f_n)_{n \in \mathcal{N}} \) is said to be \emph{non-degenerate} if \( 0 < \liminf_{n \to \infty } P(f_n(X_0^{(n)}) = 1) \leq \liminf_{n \to \infty } P(f_n(X_0^{(n)}) = 1)  < 1 \). A Boolean function \( f_n \colon \{ 0,1 \}^{n} \to \{ 0,1 \} \) is said to be transitive if for all \( i,j \in [n] \coloneqq \{ 1,2, \ldots, n \} \) there is a permutation \( \sigma \) of \( [n] \) which is such that (i) \( \sigma(i) = j \) and (ii) for all \( x \in \{ 0,1 \}^{n} \), if we define \( \sigma(x) \coloneqq (x(\sigma(k)))_{k \in [n]} \), then \( f_n(x) = f_n(\sigma(x)) \). To simplify notation, we will abuse notation slightly and say that a sequence of Boolean functions \( (f_n)_{n\geq 1} \) is transitive if \( f_n \) is transitive for each \( n \geq 1 \). 
In~\cite{js2006}, the authors show that a sufficient, but not necessary, condition for a non-degenerate sequence \( (f_n)_{n \in \mathcal{N}} \) of Boolean functions to be tame, is that \( \sup_n \mathbb{E}[C_n]< \infty \). It is natural to ask if this condition is also necessary for some natural subset of the set of all sequences of Boolean functions. This is the motivation for the following conjecture.
\begin{conjecture}[Conjecture~1.21 in~\cite{js2006}]\label{the conjecture}
For any sequences \( (p_n)_{n \geq 1} \) and \( \mathcal{N} \), if \( (f_n)_{n \in \mathcal{N}} \) is transitive, non-degenerate  and \( \lim_{n \to \infty} \mathbb{E} [C_n] = \infty \), then \( (f_n)_{n \in \mathcal{N}} \) is not tame.
\end{conjecture}

The main objective of this paper is to show that this conjecture, in its full generality, is false. This result will follow as an immediate consequence of the following theorem, which is our main result.

\begin{theorem}\label{theorem: first result}
Let  \( (p_n)_{n \geq 1} \) be a decreasing sequence which is such that 
\begin{enumerate}[(A)]
	\item \( \lim_{n \to \infty} n p_n = \infty \)
	\item \( \lim_{n \to \infty} n p_n^r = 0 \) for some \( r \geq 2 \), and
	\item for any mapping \( \phi \colon \mathbb{N} \to \mathbb{N}  \) which satisfies \( \lim_{n \to \infty} |n - \phi( n)|/(n \land \phi(n))  = 0 \), we have \( \lim_{n \to \infty} p_n / p_{\phi(n)} = 1\).
	\end{enumerate}
Then there is a sequence of positive integers \( \mathcal{N} \) and  a sequence \( (f_n)_{n \in \mathcal{N}} \) of Boolean functions, \( f_n \colon \{ 0,1 \}^{n} \to \{ 0,1 \} \), which (w.r.t\ \( (p_n)_{n \geq 1}\)) is
\begin{enumerate}[(a)]
    \item  non-degenerate,
    \item transitive,
    \item tame, and 
    \item \( \lim_{n \to \infty}\mathbb{E}[C_n] = \infty \).
\end{enumerate}
\end{theorem}

In the proof of Theorem~\ref{theorem: first result} we give  explicit examples of sequences of functions which satisfies the above conditions, and hence contradicts Conjecture~\ref{the conjecture}, for sequences \( (p_n)_{n \geq 1} \) which satisfies the assumptions above.
We remark however that this example cannot directly be extended to the case  \( 0 \ll p_n \ll 1 \). In particular, the conjecture might hold with some additional restriction on the sequence \( (p_n)_{n \geq 1} \).

\begin{remark}
The assumption that \( \lim_{n \to \infty} np_n = \infty \) is very natural. To see this, note that the number of jumps of \( (X_t)_{t \geq 0} \) in \( (0,1) \) is given by \( 2np_n(1-p_n) \), and hence if \( \limsup_{n \to \infty} np_n < \infty \), then any sequence \( (f_n)_{n \geq 1} \) of Boolean functions is tame and satisfies    \( \limsup_{n \to \infty} \mathbb{E}[C_n] < \infty \).
\end{remark}

\begin{remark}
With slightly more work, one can modify the example given in the proof of Theorem~\ref{theorem: first result} to get a sequence of functions which, in addition to satisfying (a), (b), (c) and (d) of Theorem~\ref{theorem: first result}, is also monotone in the sense that if \( x,x' \in \{ 0,1 \}^n \), \( n \in \mathcal{N} \), if \( x(i) \leq x'(i) \) for all \( i \in [n] \), then \( f(x) \leq f(x') \). 
\end{remark}

 \begin{remark}
 The same idea which is used in the proof of Theorem~\ref{theorem: first result} work in general to disprove Conjecture~\ref{the conjecture} whenever one can find a non-degenerate, tame and transitive sequence of Boolean functions. In particular, this implies that the assumption that \( \mathbb{E}[C_n] = \infty \) can be dropped from Conjecture~\ref{the conjecture}.
 \end{remark}

 With the previous remark in mind, we suggest the following modified conjecture.
 \begin{conjecture}\label{the conjecture II}
If \( (p_n)_{n \geq 1} \) satisfies \( \lim_{n \to \infty }np_n^r = \infty \) for all \( r > 0 \), and \( (f_n)_{n \geq 1} \) is transitive and  non-degenerate (w.r.t.\ \((p_n)_{n\geq 1}\)), then \( (f_n)_{n \geq 1} \) is not tame.
\end{conjecture}

\section{Proof of the main result}

\begin{definition}[Easily convinced tribes]\label{def: easily convinced tribes}
Fix \( r \geq 2 \) and \( n \geq 1 \), and let \( \ell_n \geq 2\) and \( k_n \) be positive integers with the property that \( \ell_n k_n = n \). Partition \( [n] \) into \( k_n \) sets \( S_1^{(n)}, S_2^{(n)}, \ldots, S_{k_n}^{(n)}\), each of size \( \ell_n \), and for \( x \in \{ 0,1 \}^n \), let \( g_n(x) = g_n^{(k_n,\ell_n,r)}(x) \) be equal to one exactly when there is some \( j \in \{ 1,2, \ldots, k_n \} \) such that \( \sum_{i \in S_j^{(n)}} x(i) \geq r \).
\end{definition}

Since Definition~\ref{def: easily convinced tribes} requires that \( k_n\ell_n = n \) and that \( \ell_n \geq 2 \), \( g_n^{(k_n, \ell_n, r)} \) is only well defined when \( n \) is not a prime, and we will in general only want to consider sub-sequences \( \mathcal{N} \) of \( \mathbb{N} \) which have the property that \( k_n \) and \( \ell_n \) can be chosen such that they satisfy certain growth conditions.

We will now show that we can choose \( r  \), \( (p_n)_{n \geq 1} \),  \( \mathcal{N}   \), \( (\ell_n)_{n \in \mathcal{N}} \) and  \( (k_n)_{n\in \mathcal{N}} \)  so that (a), (b) and (d) of Theorem~\ref{theorem: first result} hold.

\begin{lemma}\label{lemma: gn}
For any \( r \geq 2 \) and any decreasing sequence \( (p_n)_{n \geq 1} \) which satisfies the assumptions of Theorem~\ref{theorem: first result}, there is \( \mathcal{N} \) and sequences  \( (\ell_n)_{n \in \mathcal{N}} \), \( (k_n)_{n \in \mathcal{N}} \) of positive integers such that \( (g_n)_{n \in  \mathcal{N}} \) is 
\begin{enumerate}[(a)]
    \item non-degenerate,
    \item transitive, and
    \item tame.
\end{enumerate} 
\end{lemma}

\begin{proof}[Proof of Lemma~\ref{lemma: gn}]
Assume that there are sequences \( \mathcal{L} \), \( (\ell_n)_{n \in \mathcal{L}} \) and \( (k_n)_{n \in \mathcal{N}} \) such that
\begin{enumerate}[(i)]
    \item \( 2r < \inf \ell_n  \),
    \item \( \lim_{n \to \infty}p_n \ell_n = 0 \), and
    \item \( p_n^r \ell_n^r k_n \asymp 1 \).
\end{enumerate}
We first show that this assumption implies that the conclusions of the lemma hold, and then show that we can find sequences \( \mathcal{L} \), \( (\ell_n)_{n \in \mathcal{L}} \) and \( (k_n)_{n \in \mathcal{N}} \) with these properties.

\begin{proof}[Proof of (a)]
    \renewcommand{\qedsymbol}{}
Note first that 
  \begin{equation}\label{eq: proof eq 1}
    P(g_n(X_0^{(n)})=0) = \biggl( \sum_{i=0}^{r-1} \binom{\ell_n}{i} p_n^i (1-p_n)^{\ell_n-i}  \biggr)^{k_n}.
    \end{equation}
    
    For integers \( 0 < r < \ell \) such that \( 2 r < \ell \), define  \( T \colon \mathbb{R} \to \mathbb{R}\) by \( T(x) \coloneqq \sum_{i=0}^{r-1} \binom{\ell}{i} x^i (1-x)^{l-i} \). Then
        \begin{equation*}
        \begin{cases}
        T(x) = \sum_{i=0}^{r-1} \binom{\ell}{i} x^i (1-x)^{l-i} \cr 
        T'(x) = -\binom{\ell}{r} \cdot rx^{r-1} (1-x)^{\ell-r} \cr
        T''(x) = -\binom{\ell}{r} \cdot r(r-1)x^{r-2} (1-x)^{\ell-r}+\binom{\ell}{r} \cdot r(\ell-r)x^{r-1} (1-x)^{\ell-r-1} \cr
        \ldots  \cr
        T^{(m)}(x) = -\binom{\ell}{r} \sum_{i=1}^{m \land r\land \ell-r}  \binom{r}{i} (r)_i x^{r-i} (\ell-r)_{m-i}(1-x)^{\ell-r-i}(-1)^{m-i}    
        \end{cases}
    \end{equation*}
    and hence
    \begin{equation*} 
        \begin{cases}
        T(0) = 1 \cr 
        T^{(m)}(0) = 0 \text{ if } j = 1,2, \ldots, r-1  \cr
        T^{(r)}(0) = -(\ell)_r    .
        \end{cases}
    \end{equation*} 
    Moreover, if we assume that  \( x \in (0,1) \) and that \( \ell x < 1 \), then for all \( \xi \in (0,x) \) we have that
    \begin{align*} 
        &|T^{(r+1)}(\xi)| = \Bigl| -\binom{\ell}{r} \sum_{i=1}^{r}  \binom{r}{i} (r)_i \xi^{r-i} (\ell-r)_{r+1-i}(1-\xi)^{\ell-r-i}   (-1)^{r+1-i} \Bigr|
        \\&\qquad \leq \binom{\ell}{r}    \sum_{i=1}^{r}  \binom{r}{i} (r)_i \xi^{r-i} (\ell-r)_{r+1-i} 
        \leq \binom{\ell}{r}    \sum_{i=1}^{r}  \binom{r}{i} (r)_i \xi^{r-i} \ell^{r-i+1} 
        \\&\qquad \leq \binom{\ell}{r}    \sum_{i=1}^{r}  \binom{r}{i} (r)_i \cdot \ell
        \leq  \ell^{r+1} 2^r . 
    \end{align*}
    Applying Taylor's theorem to the right hand side of~\eqref{eq: proof eq 1}, and noting that \( 2^r < (r+1)! \), we obtain
  \begin{align*}
    &P\bigl(f_n(X_0^{(n)})=0 \bigr) =  \Bigl(1 - (\ell_n)_r p_n^r + C_n p_n^{r+1} \ell_n^{r+1}\Bigr)^{k_n}
    \end{align*} 
    where \( |C_n| < 1 \) for all \( n \). By using the inequalities \( e^{-2 x} \leq 1-x \leq e^{-x} \), valid for all \( x \in [0,1/2) \), the desired conclusion follows by applying (iii).
\end{proof}

\begin{proof}[Proof of (b)]
    \renewcommand{\qedsymbol}{}

Fix some \( n \in \mathcal{L}\) and \( i,i' \in [n] \). We need to show that there is a permutation \( \sigma \) which is such that \( \sigma(i) = i' \) and \( f_n(\sigma(x)) = f_n(x) \) for all \( x \in \{ 0,1 \}^n \). We now divide into two cases.
First, if \( i,i' \in S_m^{(n)} \) for some \( m \in [k_n] \), then we can set \( \sigma = (ii') \). 
On the other hand, if there are distinct \( m,m' \in [k_n] \) such that  \( i \in S_m^{(n)} = \{ i, i_2, \ldots, i_{\ell_n}\} \) and \( i' \in S_{m'}^{(n)} = \{ i', i_2', \ldots, i_{\ell_n}'\}\), then we can set \( \sigma = (ii') \prod_{j=2}^{\ell_n} (i_j i_j')\).
This concludes the proof of (b). 
\end{proof}
 
 \begin{proof}[Proof of (c)]
    \renewcommand{\qedsymbol}{}
Fix some \( n \in \mathcal{L} \) and note that whenever \( g_n(X_0^{(n)}) =1\), the distribution of the smallest time \( t > 0 \) at which \( g_n(X_t^{(n)})=0 \) stochastically dominates an exponential distribution with rate \( r \). From this the desired conclusion follows.
\end{proof}

To complete the proof of Lemma~\ref{lemma: gn} it now remains only to show that there are sequences \( \mathcal{N} \),  \( (\ell_n)_{n \in \mathcal{L}} \) and \( (k_n)_{n \in \mathcal{N}} \) such that (i), (ii) and (iii) hold.
To this end, for each  \( n \geq 1 \) let \(  \ell_n  \coloneqq  (np_n^r)^{-1/(r-1)}  \) and \(  k_n \coloneqq n/\ell_n = (np_n)^{r/(r-1)}\). Then one easily verifies that \( 2r < \inf \ell_n  \), \( \lim_{n \to \infty}p_n \ell_n = 0 \), and \( p_n^r \ell_n^r k_n \asymp 1 \).  
However, in general, neither \(  \ell_n \) nor \(  k_n \) need to be integers. 
To fix this problem, define
\begin{equation*}
    \begin{cases}
        \hat n \coloneqq \lceil  \ell_n \rceil \, \lceil  k_n \rceil, \cr
        \hat \ell_{\hat n} \coloneqq \lceil  \ell_n \rceil, \text{ and}\cr
        \hat k_{\hat n} \coloneqq \lceil  k_n \rceil .
    \end{cases}
\end{equation*}
Let \( { \mathcal{N}}\subseteq \mathbb{N} \) be an infinite sequence on which the mapping \( n \mapsto \hat n \) it is a bijection, and let \( \hat {\mathcal{N}} \) be its image. 
We will show that the desired properties hold for \( \hat{\mathcal{N}}\), \( (p_{\hat n})_{\hat n \in \hat {\mathcal{N}}} \), \( (\hat \ell_{\hat n})_{\hat n \in \hat {\mathcal{N}}} \) and \( (\hat k_{\hat n})_{\hat n \in \hat {\mathcal{N}}} \). To this end, note first that 
\begin{equation*}
    \inf_{\hat n \in \hat {\mathcal{N}}} \hat \ell_{\hat n}  = \inf_{ n \in  {\mathcal{N}}} \lceil \ell_n \rceil > \inf_{ n \in \mathbb{N}} \lceil \ell_n \rceil > \inf_{ n \in \mathbb{N}}  \ell_n   > 2r,
\end{equation*}
and hence (i) holds.
Next, since \( \hat n \geq n \) for each \( n \in \mathbb{N} \) and \( (p_n)_{n \geq 1} \) is decreasing, we have that \( p_{\hat n} \leq p_n \) for all \( n \in \mathbb{N} \). Using this observation, we obtain
\begin{equation*}
    \lim_{\hat n \to \infty} \hat \ell_{\hat n} p_{\hat n} \leq \lim_{\hat n \to \infty} \hat \ell_{\hat n} p_{n} = \lim_{n \to \infty}   \lceil \ell_{n} \rceil p_{n}
    = \lim_{n \to \infty}   \bigl(\ell_np_n + (\lceil \ell_{n} \rceil - \ell_n) p_{n}\bigr)=0,
\end{equation*}
and hence (ii) holds.
Finally, to see that (iii) holds, note that for any \( n \in \mathbb{N}\),
\begin{align*}
    &| n - \hat n| =  \hat \ell_{\hat n} - \hat k_{\hat n} -\ell_nk_n
    = \lceil k_n \rceil  \lceil \ell_n \rceil  - \ell_nk_n
    \\&\qquad = (\lceil k_n \rceil - k_n)(\lceil \ell_n \rceil - \ell_n) + \ell_n(\lceil k_n \rceil - k_n) +  k_n(\lceil \ell_n \rceil - \ell_n) < \ell_n + k_n +1.
\end{align*}
Since both \( \ell_n \to \infty \) and \( k_n \to \infty \) by definition, we have \( \lim_{n \to \infty} |n - \hat n|/\hat n  = 0 \), and hence by assumption, \( \lim_{n \to \infty} p_n / p_{\hat n} = 1\). This implies in particular that for \( \hat n \in \hat {\mathcal{N}} \), we have 
\begin{equation*}
    p_{\hat n}^r \hat \ell_{\hat n}^r \hat k_{\hat n} 
    \sim p_{n}^r \hat \ell_{\hat n}^r \hat k_{\hat n}  
    = p_{n}^r  \lceil \ell_{ n}\rceil^r \lceil  k_{ n}  \rceil
     = p_{n}^r  \bigl(\ell_n - (\ell_n - \lceil \ell_{ n}\rceil)\bigr)^r \bigl(k_n - (k_n-\lceil  k_{ n}  \rceil)\bigr) 
\end{equation*}
Using the assumption that \( p_n\ell_n \to 0 \), it follows that
\begin{equation*}
    p_{\hat n}^r \hat \ell_{\hat n}^r \hat k_{\hat n} 
    \sim  p_{n}^r  \bigl(\ell_n - (\ell_n - \lceil \ell_{ n}\rceil)\bigr)^r k_n
    =   p_{n}^r \ell_n^{r} k_n \sum_{i=0}^{r} \binom{r}{i}     \biggl( \frac{\lceil \ell_{ n}\rceil - \ell_n}{\ell_n}\biggr)^{r-i}   
    \sim    p_{n}^r \ell_n^{r} k_n  
    \asymp 1.
\end{equation*}
This concludes the proof.
\end{proof}

\begin{remarks}
Using essentially the same argument as in the proof of Theorem~\ref{theorem: first result}(c), one can show that for any \( r \geq 2 \), \( \mathcal{N} \), \( (\ell_n)_{n \in \mathcal{N}} \) and \( (k_n)_{n\in \mathcal N} \), we have \( \lim_{n \to\infty }\mathbb{E}[C_n(g_n)] < \infty \). This has two interesting consequences.
	\begin{enumerate}
		\item For any choice of parameters, the sequence of Boolean functions defined in~\ref{def: easily convinced tribes} does not satisfy (c) in Theorem~\ref{theorem: first result}, and hence does not provides a counter-example to Conjecture~\ref{the conjecture}. 
		\item Given a Boolean function \( f_n \colon \{ 0,1 \}^n \to \{ 0, 1 \}^n\) and \( i \in \{ 1,2 \ldots, n \} \), we let \( I_i^{(p_n)}(f_n) \) denote the \emph{influence of the \( i \)th bit} on \( f_n \) at \( p_n \), defined as the probability that resampling the \( i\)th bit of \( X_0^{(n)} \) according to \( (1-p_n)\delta_0 + p_n \delta_1 \) changes the value of \( f_n(X_0^{(n)}) \). Note that this definition agrees with the definition of influence given in~\cite{js2006}, but differs with a factor \( 2p_n(1-p_n) \) from the definition of influence used in e.g.\ \cite{bkkkl} and \cite{odonnell}.
			We let \( I^{(p_n)}(f_n) \coloneqq \sum_{i=1}^n I_i^{(p_n)}(f_n) \) and call this the \emph{total influence} of \( f_n \) at \( p_n \).
		By Proposition~1.19 in~\cite{js2006} the total influence of \( g_n \) is equal to  \(  \mathbb{E}[C_n(g_n)]  \). It thus follows from Lemma~\ref{lemma: gn} that when \( \lim_{n \to \infty}np_n^r = 0 \) for some \( r \geq 2 \),  there is a sequence of Boolean functions which is non-degenerate, transitive and have bounded total influence.  Using Proposition~1.19 in~\cite{js2006} together with the proof of Lemma~\ref{lemma: hn} below, it in fact follows that any such sequence could be used to create a counter-example to Conjecture~\ref{the conjecture} as in the proof of Theorem~\ref{theorem: first result}. By contrast, by Theorem~1 in~\cite{bkkkl}, the total influence of any non-degenerate and transitive Boolean function \( f_n \colon \{ 0,1 \}^n \to \{ 0,1 \} \) is  of order at least \( p_n^2 \log n \). In particular, this implies that when \( p_n \gg (\log n)^{-1/2}   \) there can be no sequence of non-degenerate and transitive Boolean functions with bounded total influence. This does however not exclude the possibility of a counter-example to Conjecture~\ref{the conjecture} in this regime.
	\end{enumerate}
\end{remarks}

We now want to modify the sequence \( (g_n)_{n \in \mathcal{N}} \) slightly to obtain a sequence \( (f_n)_{n \in \mathcal{N}} \) of Boolean functions which in addition to satisfying (a), (b) and (c) of Lemma~\ref{lemma: gn} also satisfies \( \lim_{n \to \infty} \mathbb{E}[C_n(f_n)] = \infty \). To this end, we first define a degenerate sequence of Boolean functions with this property. 

\begin{definition}
For each \( n \geq 1 \), let \( a_n > 0 \) and \( H_n \coloneqq np_n + a_n \sqrt{np_n(1-p_n)} \).
For \( x \in \{ 0,1 \}^n \), let \( \|x \|\coloneqq \sum_{i=1}^n x_i \) and define \( h_n(x) \coloneqq I(\| x \|  \geq H_n ) \).
\end{definition}

\begin{lemma}\label{lemma: hn}
If \( \lim_{n \to \infty} np_n = \infty \) and \( a_n = \sqrt{\log(np_n)}/2\), then \( (h_n)_{n \geq 1} \) is 
\begin{enumerate}[(a)]
    \item  degenerate, 
    \item transitive
    \item lame, and satisfies
    \item \( \lim_{n \to \infty}\mathbb{E}\bigl[C_n(h_n) \bigr] = \infty \).
\end{enumerate} 
\end{lemma}

\begin{proof}
Note first that the assumptions on \( (p_n)_{n \geq 1} \) and \( (a_n)_{n \geq 1} \) together imply that \( \lim_{n \to \infty} a_n = \infty \) and \( \sqrt{np_n}e^{-a_n^2} \to \infty\).

\begin{proof}[Proof of (a)]
    \renewcommand{\qedsymbol}{}
By definition, we have  \(  \mathbb{E}\bigl[\|X_0^{(n)}\|\bigr] = np_n \) and \( \Var \bigl(\|X_0^{(n)}\| \bigr) = np_n(1-p_n) \). Using Chebyshev's inequality, we thus obtain
\begin{align*}
    P\bigl(h_n(X_0^{(n)}) = 1\bigr) = P\Bigl(\| X_0^{(n)}\| \geq \mathbb{E}\bigl[\| X_0^{(n)}\|\bigr] + a_n \sqrt{\Var\bigl(\| X_0^{(n)}\|\bigr)}\Bigr) \leq a_n^{-2}.
\end{align*}
Since \( a_n \to \infty \), this implies that \((h_n)_{n \geq 1} \) is degenerate, which is the desired conclusion.
\end{proof}

\begin{proof}[Proof of (b)]
    \renewcommand{\qedsymbol}{}
Since for any \( x \in \{ 0,1 \}^n \), \( h_n(x) \) depends on \( x \) only through \( \| x \|\), \( (h_n)_{n \geq 1} \) is transitive.
\end{proof}

\begin{proof}[Proof of (c)]
    \renewcommand{\qedsymbol}{}
Recall that whenever \( np_n(1-p_n) \to \infty \), 
\begin{equation*}
   \biggl( \frac{\|X_t^{(n)}\|-np_n}{\sqrt{2np_n(1-p_n)}} \biggr)_{t \geq 0} \overset{\mathcal{D}}{\Rightarrow} (Z_t)_{t \geq 0},
\end{equation*}
where \( (Z_t)_{t \geq 0} \) is a so-called Ornstein-Uhlenbeck process with infinitesimal mean and variance given by \( \mu(z) = -z \) and \( \sigma^2(x) = 1 \) respectively (see e.g.\ pp.\ 170--173 in~\cite{kt}). Given \( z \in \mathbb{R} \), let \( \tau_z \) denote the first time \( t \geq 0 \) at which \( Z_t = z\), given that \( Z_0 \) is chosen according to the stationary distribution of \( (Z_t)_{t \geq 0} \).
By Corollary 1 in~\cite{rs} (see also~\cite{crs}), when \( z>0 \) is large, we have \(  \mathbb{E} [\tau_z] \sim 1/\hat h(z)\) and \( \Var (\tau_z) \sim 1/\hat h(z)^2\), where \( \hat h(z) = z \exp(-z^2/2) /\sqrt{2 \pi}\). By the Paley-Zygmund inequality, this implies that for any finite time \( t > 0 \),  \( \lim_{z \to \infty} P(\tau_{z} > t) = 1 \). This implies in particular that \( (h_n)_{n \geq 1} \) is lame whenever \( a_n \to \infty \), completing the proof of (c).
\end{proof}

\begin{proof}[Proof of (d)] 
    \renewcommand{\qedsymbol}{}
By Proposition 1.19 in~\cite{js2006}, for each \( n \geq 1 \) we have
\begin{equation*}
    \mathbb{E}[C_n(h_n)] = \sum_{i=1}^n I_i^{(p_n)}(h_n) 
\end{equation*}
where \( I_i^{(p_n)}(h_n) \) is the so-called influence of the \( i \)th bit on \( h_n \) at \( p_n \), defined as the probability that resampling the \( i\)th bit of \( X_0^{(n)} \) according to \( (1-p_n)\delta_0 + p_n \delta_1 \) changes the value of \( h_n(X_0^{(n)}) \). Using this result, we obtain
\begin{align*}
    &\mathbb{E}[C_n(h_n)] = nI_1^{(p_n)}(h_n) = n \Biggl( P\bigl(\|X_0^{(n)} \|=H_n-1\bigr) \cdot \frac{n-(H_n-1)}{n} \cdot p_n + P\bigl(\|X_0^{(n)} \|=H_n\bigr) \cdot \frac{H_n}{n} \cdot (1-p_n) \Biggr)
    \\&\qquad =n \Biggl( \binom{n}{H_n-1} p_n^{H_n-1}(1-p_n)^{n-H_n+1}  \cdot \frac{n-(H_n-1)}{n} \cdot p_n + \binom{n}{H_n} p_n^{H_n}(1-p_n)^{n-H_n}  \cdot \frac{H_n}{n} \cdot (1-p_n) \Biggr)
    \\&\qquad= 2H_n   \binom{n}{H_n} p_n^{H_n}(1-p_n)^{n-H_n+1}.
\end{align*}  
Using Stirling's formula, it follows that
\begin{align*}
    &\mathbb{E}[C_n(h_n)]   
    \sim 2 \sqrt{np_n} \cdot    \frac{e^{-a_n^2}}{\sqrt{2 \pi}}.
\end{align*} 
In particular, if \( \sqrt{np_n}e^{-a_n^2} \to \infty \), then \(\lim_{n \to \infty}\mathbb{E}[C_n(h_n)] = \infty \). This completes the proof of (d).
\end{proof}
\end{proof}

We are now ready to give a proof of our main result.

\begin{proof}[Proof of Theorem~\ref{theorem: first result}]
Fix some \( r \geq 2 \) and sequences \( \mathcal{N} \), \( (\ell_n)_{n \in \mathcal{N}} \), \( (k_n)_{n \in \mathcal{N}} \) and \( (a_n)_{n \geq 1} \) so that the assumptions of Lemmas~\ref{lemma: gn} and Lemma~\ref{lemma: hn} both hold. For \( n \in \mathcal{N} \) and \( x \in \{ 0,1 \}^n \), let  \( H_n \coloneqq np_n + a_n\sqrt{np_n(1-p_n)} \) and define 
\begin{equation}
    f_n(x) \coloneqq  g_n(x)I(\|x\| < H_n-1) + I(\|x \| \geq H_n) = g_n(x)I(\|x\| < H_n-1) + h_n(x).
\end{equation}
Note that \( f_n (x)\) and \( h_n (x)\) agree whenever  \( \| x \| \geq H_n-1 \).
Combining Lemma~\ref{lemma: gn} and Lemma~\ref{lemma: hn},  the desired conclusion now immediately follows. 
\end{proof}

\section*{Acknowledgements}

The author would like to thank the anonymous referee for several useful comments.
The author acknowledges support from the European Research Council, Grant Agreement No. 682537.

\end{document}